\documentclass{amsart}

\usepackage{amssymb} \usepackage{amsfonts} \usepackage{amsmath}
\usepackage{amsthm} \usepackage{epsfig} 
\usepackage{color}
\usepackage{amscd}
\usepackage[all]{xy}
\usepackage{graphicx}
\usepackage{pinlabel}
\usepackage{url}

\usepackage{pgf,tikz}
\usetikzlibrary{arrows}

\usepackage[bookmarks=true,colorlinks=true,urlcolor=black,
linkcolor=black,citecolor=black,pagebackref=false]{hyperref}

\newtheorem{lemma}{Lemma}[section]
\newtheorem{teo}[lemma]{Theorem}
\newtheorem{prop}[lemma]{Proposition}
\newtheorem{cor}[lemma]{Corollary}

\theoremstyle{definition}

\newtheorem{quest}[lemma]{Question}

\theoremstyle{remark}
\newtheorem{rem}[lemma]{Remark}

\newcommand{\Iso}{{\rm Isom}}

\newcommand{\matR} {\ensuremath {\mathbb{R}}}

\newcommand{\matZ} {\ensuremath {\mathbb{Z}}}

\newcommand{\matH} {\ensuremath {\mathbb{H}}}

\newcommand{\Vol} {\ensuremath {{\rm Vol}}}

\author{Bruno Martelli}
\address{Dipartimento di Matematica, Largo Pontecorvo 5, 56127 Pisa, Italy}
\email{martelli at dm dot unipi dot it}

\author{Stefano Riolo}
\address{Institut de math\'ematiques, Rue Emile-Argand 11, 2000 Neuch\^atel, Switzerland}
\email{stefano dot riolo at unine dot ch}

\author{Leone Slavich}
\address{Dipartimento di Matematica, Via Ferrata 5, 27100 Pavia, Italy}
\email{leone dot slavich at gmail dot com}

\thanks{S. R. was supported by the SNSF project no. PP00P2-170560, and thanks the Mathematics Department of the University of Pisa for the hospitality.}

\title[Convex plumbings in closed hyperbolic 4-manifolds]{Convex plumbings \\ in closed hyperbolic 4-manifolds}

\begin{document}

\begin{abstract}
We show that every plumbing of disc bundles over surfaces whose genera satisfy a simple inequality may be embedded as a convex submanifold in some closed hyperbolic four-manifold. In particular its interior has a geometrically finite hyperbolic structure that covers a closed hyperbolic four-manifold.
\end{abstract}

\maketitle

\section{Introduction}

We study the following general question. All the manifolds in this paper are assumed implicitly to be smooth, connected, and oriented, unless otherwise stated.

\begin{quest} \label{main:quest}
Let $M$ be a compact smooth $n$-manifold with non-empty boundary. Is there a closed hyperbolic $n$-manifold $W$ containing $M$ as a convex submanifold?
\end{quest}

Here \emph{convex} means that every arc in $M$ is homotopic (relative to its endpoints) in $M$ to a geodesic. This is in fact a local property of the boundary of $M$, since $M$ is convex if and only if it is locally convex  
\cite[Chapter I.1.3]{CEG}.

In dimension $n=2$ the answer to this question is positive for any given surface $M$ with boundary. In dimension $n=3$, it is positive precisely when $M$ is irreducible and algebraically atoroidal (that is, $\pi_1(M)$ does not contain $\matZ \times \matZ$). This is a manifestation of geometrisation, see Remark \ref{many:rem}.

Why are we interested in convex submanifolds $M\subset W$ in closed hyperbolic manifolds? One motivation is that being convex gives $M$ some privileges. The embedding is necessarily $\pi_1$-injective, and the cover $\tilde W \to W$ associated to the subgroup $\pi_1 M < \pi_1 W$ is a geometrically finite complete hyperbolic manifold diffeomorphic to the interior of $M$. So in particular the interior of $M$ has a geometrically finite hyperbolic structure that covers the closed hyperbolic $W$. See Section \ref{convex:subsection}.


We would like to investigate Question \ref{main:quest} in the higher dimensions $n \geq 4$, where our knowledge of the topology of hyperbolic manifolds is embarrassingly poor. The main contribution of this paper is to furnish a family of examples in dimension 4. 

\subsection*{Plumbings}
Recall that a \emph{plumbing graph} is a graph where every node is assigned a pair $(e_i,g_i)$ of integers with $g_i\geq 0$, and every edge is given a sign $\varepsilon_j = \pm 1$. Loops and multiple edges connecting two nodes are allowed. Given a plumbing graph, one may construct an oriented compact four-manifold $M$, called \emph{plumbing}, by taking for each vertex the disc bundle with Euler number $e_i$ over the surface with genus $g_i$, and performing for each edge a plumbing along the corresponding bundles that creates a transverse intersection of the base surfaces with sign $\varepsilon_j$. See \cite[Section 4.6.2]{GS} for more details. 

The regular neighbourhood of a generically
immersed closed (possibly disconnected) surface in a four-manifold is a disjoint union of plumbings. Plumbings are ubiquitous in dimension four and it is natural to ask whether they can be embedded as convex subsets in some closed hyperbolic four-manifold. We prove here the following.

\begin{teo} \label{main:teo}
Let $M$ be a plumbing whose graph satisfies
$$g_i \geq 2\big(|e_i| + v_i +1)$$
at every vertex, where $v_i$ is its valence. The manifold $M$ is contained as a convex subset in a closed hyperbolic four-manifold $W$ with
$$\Vol(M) < C_1 \cdot \sum_i g_i \quad \mbox{and} \quad \Vol(W) < C_2 \cdot \Vol(M),$$ 
for some $C_1,C_2 > 0$ independent of the plumbing graph. The same result holds for any manifold $M$ that is a $\partial$-connected sum of plumbings satisfying the above requirement.
\end{teo}

This result is already new (to the best of our knowledge) when the plumbing graph is a point. In this case the manifold $M$ is a disc bundle, and the theorem says that if $g\geq 2(|e| + 1)$ then $M$ embeds convexly in some closed hyperbolic four-manifold $W$. 


If the graph consists of a vertex and an edge, then $M$ is a self-plumbed disc bundle, and the sufficient condition to embed $M$ convexly in some closed hyperbolic four-manifold is $g\geq 2(|e|+3)$. More generally, if $M$ is constructed by self-plumbing $k$ times a disc bundle (with any signs $\pm 1$), the condition is $g\geq 2(|e|+2k+1)$. 

We note that the inequality $g_i \geq 2$ is a necessary condition in general, because on a convex compact hyperbolic manifold $M$ we have $\pi_2(M) = \{0\}$ and $\matZ \times \matZ \not <\pi_1(M)$, so neither spheres nor tori are allowed in the plumbing.

Theorem \ref{main:teo} implies the following.
\begin{cor} \label{main:cor}
For every symmetric integer matrix $Q$ there is a boundary connected sum of plumbings $M$ with intersection form $Q$ that embeds as a convex submanifold into a closed hyperbolic 4-manifold $W$.
\end{cor}
\begin{proof}
If $Q$ is a $k\times k$ matrix, pick a plumbing graph with $k$ vertices such that $e_j = Q_{jj}$ and there are $|Q_{ij}|$ edges connecting the vertices $i \neq j$, all decorated with ${\rm sign}(Q_{ij})$. Pick any $g_i$ that satisfies the inequality of Theorem \ref{main:teo}. We get a plumbing $M$ with intersection form $Q$ that embeds convexly in a closed hyperbolic 4-manifold. If $M$ is disconnected (that is, if $Q$ is reducible), we priorly connect it with some $\partial$-connected sums: this operation is allowed by Theorem \ref{main:teo}.
\end{proof}

\begin{cor} \label{main:cor2}
For every symmetric integer matrix $Q$ there is a boundary connected sum of plumbings $M$ with intersection form $Q$, whose interior admits a geometrically finite complete hyperbolic structure that covers a closed hyperbolic manifold $W$. 
\end{cor}

All these manifolds $M$ and $W$ are constructed by assembling right-angled 120-cells. In particular all the closed manifolds $W$ constructed here cover the same Coxeter simplex orbifold $\{5,3,3,4\}$. 

\subsection*{Related work}
The first complete hyperbolic structures on the interior of some disc bundles over surfaces with genus $g>0$ and Euler number $e\neq 0$ were exhibited by Gromov -- Lawson -- Thurston \cite{GLT} and Kapovich \cite{K}. Kuiper \cite{Ku} built specimens for all $e,g$ with $|e| \leq 2/3(g-1)$, and then Luo \cite{L} for all $|e| \leq g-1$. More recently Anan'in and Chiovetto \cite{AC} constructed some examples with $|e| = 6/5(g-1)$. Gromov -- Lawson -- Thurston \cite{GLT} conjectured that $|e| \leq 2(g-1)$ in all cases.

Our contribution is to show that when $|e| \leq 1/2(g-2)$ the disc bundle embeds convexly in a closed hyperbolic four-manifold, and this is a stronger property than having a complete hyperbolic structure in its interior. 


In Theorem \ref{main:teo}, if $e_i$ is odd for some $i$ then both $M$ and $W$ have odd intersection form and are hence both non-spin. We find here many examples of non-spin closed hyperbolic 4-manifolds $W$. We constructed the first such manifolds recently in \cite{MRS}, and the techniques employed here are an extension of these.

We note that any compact hyperbolic manifold $M$ with geodesic boundary embeds convexly in a closed hyperbolic manifold $W$, constructed simply by mirroring $M$ along $\partial M$. Such manifolds $M$ exist in all dimensions, also with connected boundary \cite{LR}. Some explicit examples were constructed in \cite{KMT} using the right-angled 120-cell. The holonomy representations of these manifolds are locally rigid in dimension $n\geq 4$, see \cite{KS}. A plumbing cannot have a hyperbolic metric with geodesic boundary because its boundary is a graph manifold and hence does not admit any hyperbolic metric.

A comprehensive survey on higher dimensional Kleinian groups is \cite{K2}, a shorter one on finite-volume hyperbolic 4-manifold is \cite{Ma}.

\subsection*{Outline of the construction} 
We define for every triple of integers $e,i,g$ with $i\geq 0$ and
$$g \geq 2\big(|e|+i+1\big)$$
an oriented hyperbolic 4-manifold with right-angled corners $M^{e,i,g}$. The manifold $M^{e,i,g}$ is diffeomorphic (after smoothing its corners) to the disc bundle over the genus-$g$ surface with Euler number $e$. It is tessellated into right-angled 120-cells and contains $i$ \emph{islands}: these are some safety zones that will be used to perform the plumbings. The construction of $M^{e,i,g}$ follows the main theme of \cite{MRS} with some modifications. 

We then plumb the 4-manifolds with corners $M^{e,i,g}$ as prescribed by the given plumbing graph. The resulting plumbing $M$ has again the structure of a hyperbolic four-manifold with right-angled corners. By colouring and mirroring its facets we embed it as a convex submanifold into a closed hyperbolic four-manifold $W$.

The hyperbolic manifolds with right-angled corners $M^{e,i,g}$, and then $M$ and $W$, are all tessellated into a certain number of right-angled 120-cells. The base surfaces of the disc bundles form altogether an immersed surface $S \subset M$, pleated along some edges, contained in the 2-skeleton of $M$. The immersed surface $S$ is tessellated into right-angled pentagons.

\subsection*{Further research} We think that the techniques introduced here may be extended to construct many more compact hyperbolic 4-manifolds with right-angled corners, and hence many more compact 4-manifolds that embed convexly in some closed hyperbolic 4-manifold. 

Another natural research theme would consist of studying the deformations of the convex hyperbolic structures constructed in this paper, and their degenerations, as it has been done fruitfully in dimension 3 in the last decades.

\subsection*{Acknowledgements}
We thank Steven Tschantz for producing and sharing the pictures in Figure \ref{8dodecaedri_colorato:fig} and \ref{8dodecaedri_new:fig}. 

\section{Preliminaries}

\subsection{Convex submanifolds} \label{convex:subsection}
We recall some well-known facts on convex hyperbolic manifolds. We refer to \cite[Chapter I.1.3]{CEG} for a detailed introduction.

A smooth manifold $M$ with boundary is \emph{hyperbolic} if it has constant sectional curvature $-1$. Equivalently, it is locally isometric to some $n$-submanifold with boundary of $\matH^n$.

A connected hyperbolic manifold $M$ with boundary is \emph{convex} if every arc in $M$ is homotopic (relative to its endpoints) to a geodesic. This is a local property: the manifold $M$ is convex if and only if it is locally convex, that is every point has a convex neighbourhood \cite[Corollary I.1.3.7]{CEG}. 

A connected hyperbolic manifold $M$ with boundary is complete and convex if and only if the developing map $D\colon \tilde M \to \matH^n$ is a diffeomorphism onto a convex complete submanifold $C \subset \matH^n$, see \cite[Proposition I.1.4.2]{CEG}. In this case we may see $M$ directly as $M=C/_\Gamma$ for some discrete $\Gamma < \Iso(\matH^n)$. 

We then use the convexity of $C$ to show that $\Gamma$ acts freely. 
Suppose by contradiction that $\Gamma$ fixes some $x \in \mathbb{H}^n$. Let $\pi(x)$ denote the closest-point projection of $x$ to $C$. Since the map $\pi$ is $\Gamma$-equivariant, $\pi(x)$ will be a fixed point for the action of $\Gamma$ on $C$, which is absurd. 

Therefore $M$ naturally embeds isometrically in a unique complete hyperbolic manifold $\hat M = \matH^n/_\Gamma$ of the same dimension without boundary, such that the inclusion $M\hookrightarrow \hat M$ induces an isomorphism of fundamental groups \cite[Theorem I.2.4.1]{CEG}. We call $\hat M$ the \emph{extension} of $M$.

\begin{prop}
Let $M$ be a compact convex hyperbolic manifold with boundary.
Its extension $\hat M$ is geometrically finite and diffeomorphic to the interior of $M$. 
\end{prop}
\begin{proof} 
For every $x\in \partial C$ and $t\in [0,+\infty)$ we let $f(x,t) \in \matH^n$ be the point reached at time $t$ by the unit speed geodesic starting from $x$ orthogonally to $\partial C$ and directed outside $C$. We get a diffeomorphism $f\colon \partial C \times [0,+\infty) \to \matH^n \smallsetminus C$. The leaf $f(\partial C \times \{t\})$ consists of all the points at distance $t$ from $C$. Since $f$ is $\Gamma$-equivariant, the function descends to a diffeomorphism $f\colon \partial M \times [0,\infty) \to \hat M \smallsetminus M$. Therefore $\hat M$ is diffeomorphic to the interior of $M$.

The extension $\hat M$ of $M$ is geometrically finite, since its convex core is contained in the compact convex submanifold $M$.
\end{proof}

\begin{prop}
Let $M \subset W$ be a convex compact submanifold of a complete hyperbolic manifold $W $ without boundary of the same dimension. The induced map $\pi_1 M \to \pi_1 W$ is injective and the covering of $W$ induced by the subgroup $\pi_1 M < \pi_1 W$ is isometric to the extension of $M$. 
\end{prop}
\begin{proof}
The counterimage of $M$ in the universal cover $\tilde W = \matH^n$ of $W$ is a disjoint union of convex submanifolds. Since convex submanifolds in $\matH^n$ are contractible and in particular simply connected, we deduce that $M$ is $\pi_1$-injective in $W$. The covering of $W$ determined by the subgroup $\pi_1 M < \pi_1 W$ is the extension of $M$ by construction. 
\end{proof}

\subsection{Hyperbolic manifolds with right-angled corners}
We recall some of the terminology and techniques introduced in \cite{MRS}.

We represent the hyperbolic space $\matH^n$ via the disc model $B^n\subset \matR^n$. Let $P\subset B^n$ be the intersection of the (pairwise orthogonal) half-spaces $x_1, \ldots, x_n \geq 0$. 
A \emph{hyperbolic manifold with right-angled corners} is a topological $n$-manifold $M$, possibly with boundary, equipped with an atlas in $P$ and transition maps that are restrictions of isometries of $\matH^n$. We sometimes drop the words ``right-angled'' from the definition and simply call $M$ a manifold with corners. 

The boundary $\partial M$ of a manifold with corners $M$ is naturally stratified into connected closed $k$-dimensional strata called \emph{faces}, that we call \emph{vertices}, \emph{edges}, and \emph{facets} if $k=0,1$ and $n-1$, respectively.
Every face is abstractly itself a hyperbolic $k$-manifold with corners; note that a face may not be embedded, because it may be incident multiple times to the same lower-dimensional face. 

A hyperbolic manifold $M$ with corners may also be defined as a hyperbolic orbifold with mirrors, with isotropy groups $(\matZ_2)^k$ generated by $k$ reflections along orthogonal hyperplanes. We will not need this interpretation, however. Hyperbolic manifolds with geodesic boundary and right-angled polytopes are particular kinds of hyperbolic manifolds with corners. 

Let $M$ be a (possibly disconnected) hyperbolic manifold with corners. If we glue isometrically two disjoint embedded facets of $M$, we get a new hyperbolic manifold with corners. This is a crucial property. For instance, we may choose some disjoint embedded facets of $M$ and \emph{mirror} $M$ along them (that is, take two copies of $M$ and identify isometrically the pairs of selected facets).

\begin{prop} \label{mirror:prop}
Every compact connected hyperbolic manifold $M$ with right-angled corners, whose facets are all embedded, is contained in a closed connected hyperbolic manifold $W$ of the same dimension.
\end{prop}
\begin{proof}
We construct $W$ from $M$ by colouring and mirroring.
Assign to each facet of $M$ a colour in $\{1,\ldots, k\}$, so that adjacent facets have different colours. (For instance, assign distinct colours to distinct facets.)
Mirror $M$ iteratively along the facets coloured with $1, 2, \ldots, k$. We end up with a closed hyperbolic $W$ tessellated into $2^k$ copies of $M$. In fact we get an orbifold covering $W \to M$.

This is equivalent to taking $2^k$ copies $M_{a_1,\ldots, a_k}$ of $M$ with $a_i \in \{0,1\}$, and identifying every facet in $M_{a_1,\ldots,a_k}$ coloured with $c\in \{1,\ldots,k\}$ with the corresponding facet in $M_{a_1,\ldots, a_{c-1}, 1-a_c, a_{c+1}, \ldots, a_k}$.

If $M$ is compact, oriented and connected, then $W$ also is by construction.
\end{proof}

By smoothing its boundary we can transform every hyperbolic manifold with corners into a convex hyperbolic manifold. By combining this fact with Proposition \ref{mirror:prop} we get a class of manifolds for which Question \ref{main:quest} has a positive answer.

\begin{cor}
Let $M$ be a compact connected hyperbolic $n$-manifold with right-angled corners, whose facets are all embedded. By smoothing 
$\partial M$, we get a smooth manifold that embeds as a convex submanifold into a closed hyperbolic $n$-manifold.
\end{cor}

\begin{rem} \label{many:rem}
For a given compact manifold $M$ with non-empty boundary, the property of having a hyperbolic structure with right-angled corners may look more restrictive than having a convex hyperbolic structure. Contrary to that impression, in dimension 2 and 3 any compact manifold $M$ with boundary that has a convex hyperbolic structure also has a hyperbolic structure with right-angled corners. 

Indeed every compact surface with boundary has both structures, and a consequence of geometrisation is that in dimension 3 the manifold $M$ has any of the two structures if and only if $M$ is irreducible and algebraically atoroidal (that is, $\pi_1(M)$ does not contain $\matZ \times \matZ$). These conditions are certainly necessary for $M$ having a convex hyperbolic structure, and are also sufficient to equip $M$ with a hyperbolic structure with right-angled corners: 
it suffices to decorate $\partial M$ with a sufficiently complicated trivalent graph $\Gamma \subset \partial M$ to ensure that $M$ has a hyperbolic structure with right-angled corners bent precisely at $\Gamma$, see \cite[Page 83 and Proposition 7.2]{O}.
\end{rem}

\subsection{Regular right-angled polytopes and thickenings} \label{thickenings:subsection}
In dimension 3 the abundance of right-angled polytopes is regulated by Andreev's Theorem. In dimension 4 our knowledge is much more limited, and the main tools used in the literature to construct manifolds are the regular right-angled polytopes: the ideal 24-cell and the compact 120-cell. Their facets are ideal right-angled octahedra and compact right-angled dodecahedra, respectively. The 2-faces of the latter are right-angled pentagons.

The existence of such regular polytopes allows us to define a \emph{thickening} of manifolds with right-angled corners in particular cases, as follows. This procedure was first described in \cite{M}. 

Suppose that $N$ is a hyperbolic $n$-manifold with right-angled corners, tessellated into right-angled regular pentagons, dodecahedra, or ideal octahedra. Of course we have $n=2$, 3, or 3, respectively. The \emph{thickening} of $N$ is the hyperbolic $(n+1)$-manifold with right-angled corners $M$, defined from $N$ by attaching to each pentagon (or dodecahedron, octahedron) $P$ two regular right-angled dodecahedra (or 120-cells, 24-cells), one ``above'' and the other ``below'' $P$. Two dodecahedra (or 120-cells, 24-cells) attached from the same side (above or below) to two pentagons (or dodecahedra, octahedra) $P_1$ and $P_2$ that intersect in some edge (or face) $f$ should be attached correspondingly along their faces (or facets) incident to $f$ via the unique possible isometry that matches with $f$. This is possible because all the objects involved are regular, so every self-isometry of any facet extend to a self-isometry of the object.

The thickening $M$ of $N$ is a hyperbolic $(n+1)$-manifold with corners that deformation retracts onto $N$. If $N$ is tessellated into $k$ pentagons (or dodecahedra, octahedra), then $M$ is tessellated into $2k$ dodecahedra (or 120-cells, 24-cells). The facets of $M$ are of three kinds: the \emph{vertical} ones that contain (and correspond to) the facets of $N$, and the \emph{top} and \emph{bottom} ones that are contained in the dodecahedra (or 120-cells, 24-cells) that were attached above or below, and are not adjacent to the original $N$.

\section{The construction}
We prove here Theorem \ref{main:teo}, expanding some of the ideas of \cite{MRS}.

We define for every triple of integers $e,i,g$ with $i\geq 0$ and
$$g \geq 2\big(|e|+i+1\big)$$
an oriented hyperbolic 4-manifold with corners $M^{e,i,g}$. The manifold $M^{e,i,g}$ is diffeomorphic (after smoothing its corners) to the disc bundle over the genus-$g$ surface with Euler number $e$. It is tessellated into right-angled 120-cells and contains $i$ \emph{islands}, some zones (to be defined below) that will be used to perform the plumbings. Its facets are all embedded, so that Proposition \ref{mirror:prop} can be applied.

To improve clarity, we subdivide the construction of $M^{e,i,g}$ in some steps in the next sections. We perform the plumbing in Section \ref{sec:plumbing}, and conclude the proof of Theorem \ref{main:teo} in Section \ref{sec:volume} by estimating the volume of $M$ and $W$. 

\subsection{The surface with corners $\Sigma$} \label{sec:Sigma}
The starting point of our construction is the hyperbolic surface with corners $\Sigma$ tessellated into 8 right-angled pentagons shown in Figure \ref{sigma_new:fig}-(left). The surface $\Sigma$ is topologically a torus with one hole. Its boundary has four vertices, two edges of length $2\ell$, and two edges of length $4\ell$, where $\ell$ is the length of the side of the right-angled regular pentagon.

The 1-skeleton of $\Sigma$ contains a $\theta$-graph $\Theta$, onto which $\Sigma$ deformation retracts. The $\theta$-graph is the union of three oriented curves $\gamma_1, \gamma_2, \gamma_3$, whose sum $\gamma_1+ \gamma_2 +\gamma_3=0$ vanishes homologically. The curves $\gamma_0$ and $\gamma_1$ are shown in Figure \ref{sigma_casi_new:fig}. 

We assign to $\Sigma$ the orientation of Figure \ref{sigma_new:fig}.
There is an orientation-reversing isometry of $\Sigma$ sending $\gamma_1$ to $\gamma_2$.

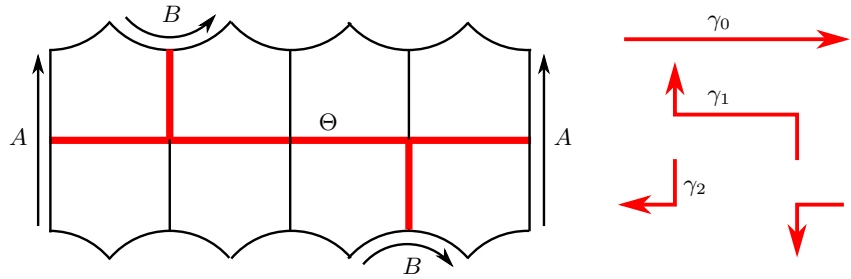
\begin{figure}
 \begin{center}
 \labellist
\small\hair 2pt
\pinlabel $A$ at -3 45
\pinlabel $A$ at 178 45
\pinlabel $B$ at 48 86
\pinlabel $B$ at 128 2
\pinlabel $\Theta$ at 100 50
\pinlabel $\gamma_0$ at 230 83
\pinlabel $\gamma_1$ at 230 58
\pinlabel $\gamma_2$ at 222 28
\endlabellist
  \includegraphics[width = 11 cm]{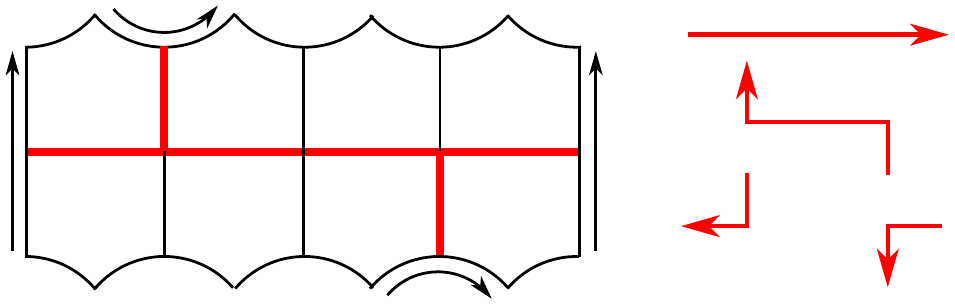}
 \end{center}
 \caption{A hyperbolic holed torus with corners $\Sigma$ tessellated into 8 right-angled pentagons. Edges marked with the same letters should be identified isometrically following the arrows (left). The holed torus deformation retracts onto a $\theta$-graph $\Theta$, drawn in red, which is in turn the union of three oriented simple closed curves $\gamma_0$, $\gamma_1$, $\gamma_2$ (right).}
 \label{sigma_new:fig}
\end{figure}

\begin{figure}
 \begin{center}
 \labellist
\small\hair 2pt
\pinlabel $A$ at -3 45
\pinlabel $A$ at 178 45
\pinlabel $B$ at 48 86
\pinlabel $B$ at 128 2
\pinlabel $A$ at 214 45
\pinlabel $A$ at 395 45
\pinlabel $B$ at 265 86
\pinlabel $B$ at 345 2
\endlabellist
  \includegraphics[width = 12.5 cm]{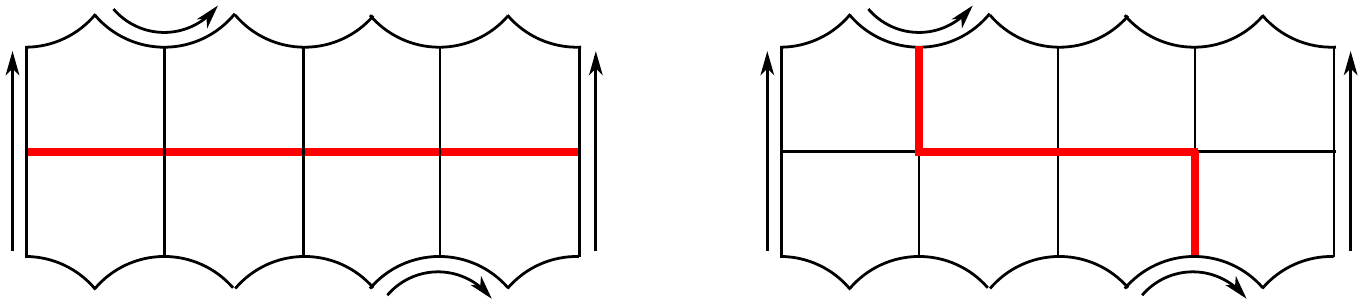}
 \end{center}
 \caption{The curves $\gamma_0$ and $\gamma_1$ in $\Sigma$.}
 \label{sigma_casi_new:fig}
\end{figure}
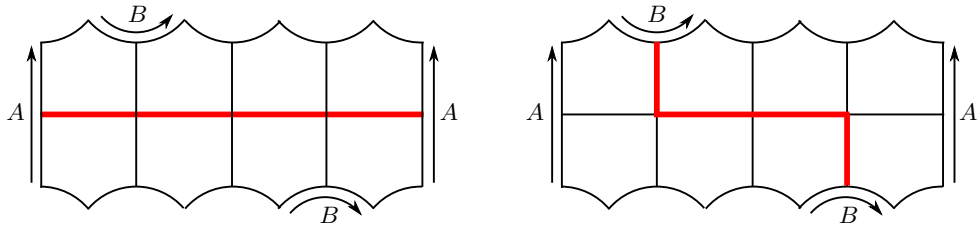

\subsection{The 3-manifold with corners $N$} \label{sec:N}
We now construct an oriented hyperbolic 3-manifold with corners $N$ that has one face isometric to $\Sigma$. 

Figure \ref{8dodecaedri_colorato:fig} shows a right-angled polyhedron $P \subset \matH^3$ obtained by taking the eight dodecahedra adjacent to a single vertex $v$ in the tessellation of $\matH^3$ into right-angled dodecahedra. The polyhedron $P$ has the symmetries of a cube with centre $v$.

\begin{figure}
 \begin{center}
  \includegraphics[width = 5 cm]{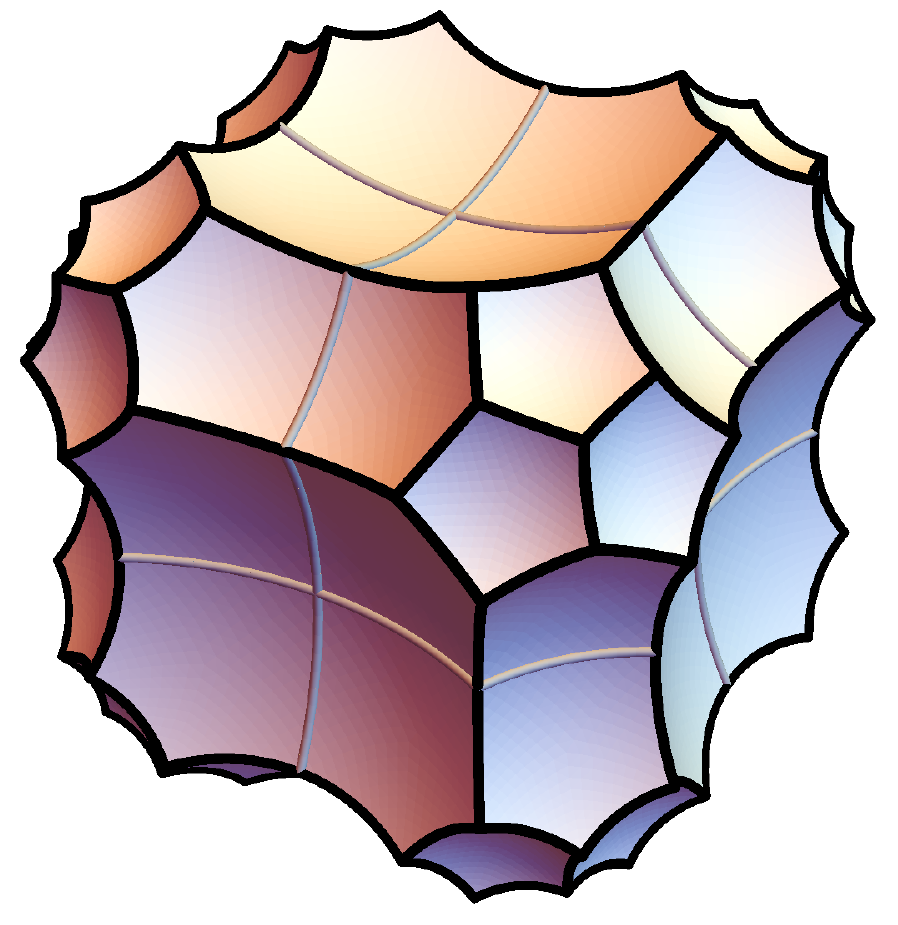}
 \end{center}
 \caption{A right-angled hyperbolic polyhedron $P$ tessellated into eight right-angled dodecahedra.}
 \label{8dodecaedri_colorato:fig}
\end{figure}

\begin{figure}
 \begin{center}
  \includegraphics[width = 12 cm]{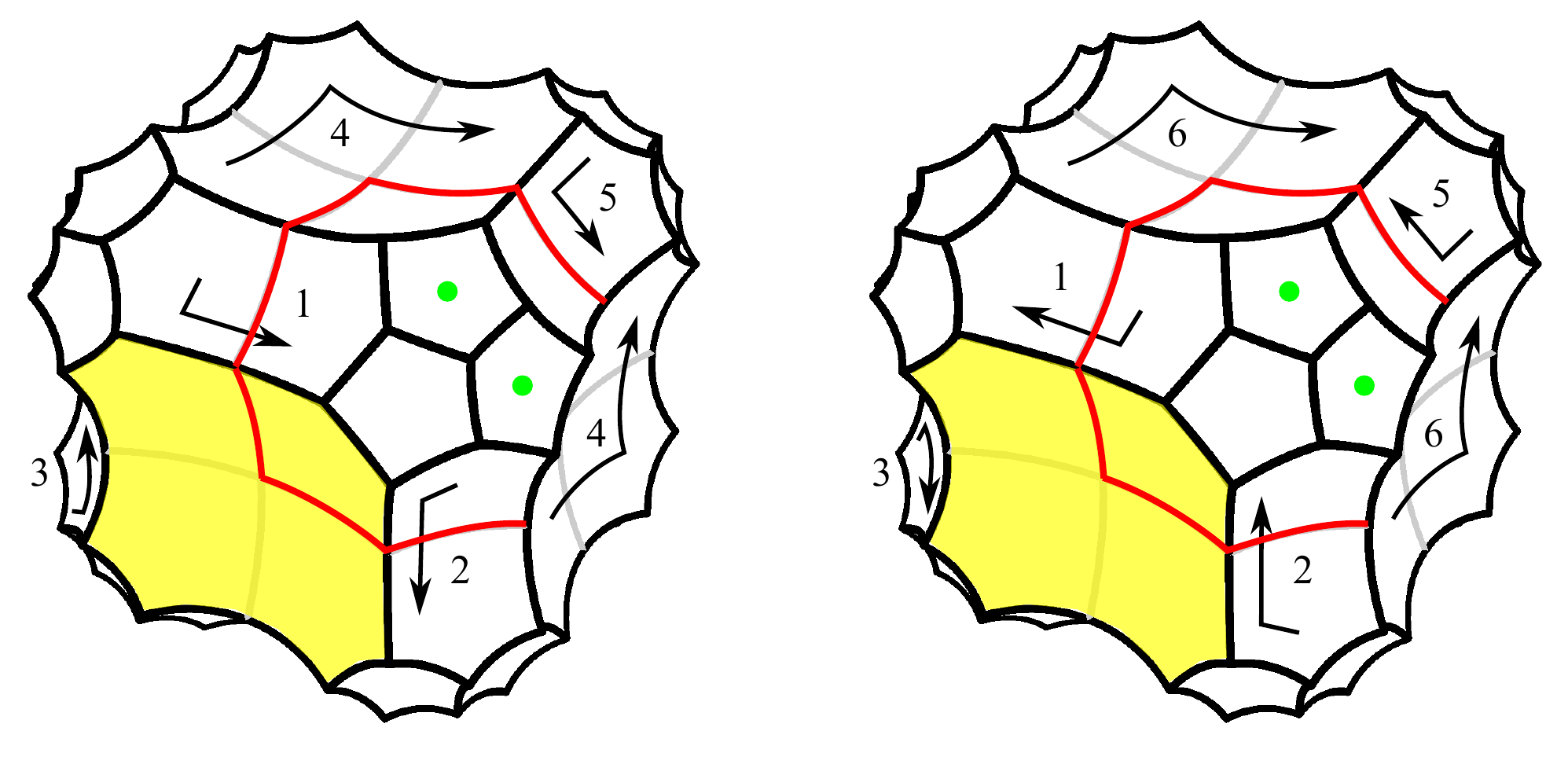}
 \end{center}
 \caption{Two copies of $P$. We identify isometrically six pairs of facets, numbered here as $1,\ldots, 6$, following the arrows. The result is a hyperbolic 3-manifold with corners that contains $\Sigma$ as a face (made of the two yellow octagons). Unfortunately the hyperbolic 3-manifold with corners contains a non-embedded face (marked here with green dots), so a more complicated construction is needed. }
 \label{8dodecaedri_new:fig}
\end{figure}

We take two copies of $P$ and identify isometrically some of their faces as prescribed by Figure \ref{8dodecaedri_new:fig}. Six pairs of faces, numbered from 1 to 6, are identified. The result is a hyperbolic 3-manifold with corners. Note that the two octagonal yellow faces shown in the figure are glued along their sides, to form a face isometric to $\Sigma$. 

This 3-manifold with corners just constructed is almost fine for our purposes, except that unfortunately it has some non-embedded faces, a fact that we want to avoid. 
Indeed, the four pentagons marked with a green dot in Figure \ref{8dodecaedri_new:fig} are attached along their boundaries to form a single non-embedded octagonal face. We already faced this issue in \cite{MRS}. There, we solved it by mirroring the polyhedron along some faces. Here we follow another strategy that is more suited to the present setting.

To resolve this problem we use a bigger polyhedron $Q \supset P$ instead of $P$, built as follows. Let $e$ be the edge separating the two green-dotted faces in $P$. We attach three dodecahedra to $e$, in the only possible way: we first attach one dodecahedron to each green-dotted face, and then a third one to cap off the resulting forbidden concave angle of 270 degrees at $e$. This operation enlarges $P$ to a bigger convex polyhedron in $\matH^3$, tessellated into $8+3=11$ dodecahedra. The edge $e$ lies in the interior of the new polyhedron. By a careful analysis we find that the boundary pattern of the corners changes as shown in Figure \ref{enlargement:fig}.

\begin{figure} 
 \begin{center}
  \labellist
\small\hair 2pt
\pinlabel $e$ at 140 240
\endlabellist
  \includegraphics[width = 10 cm]{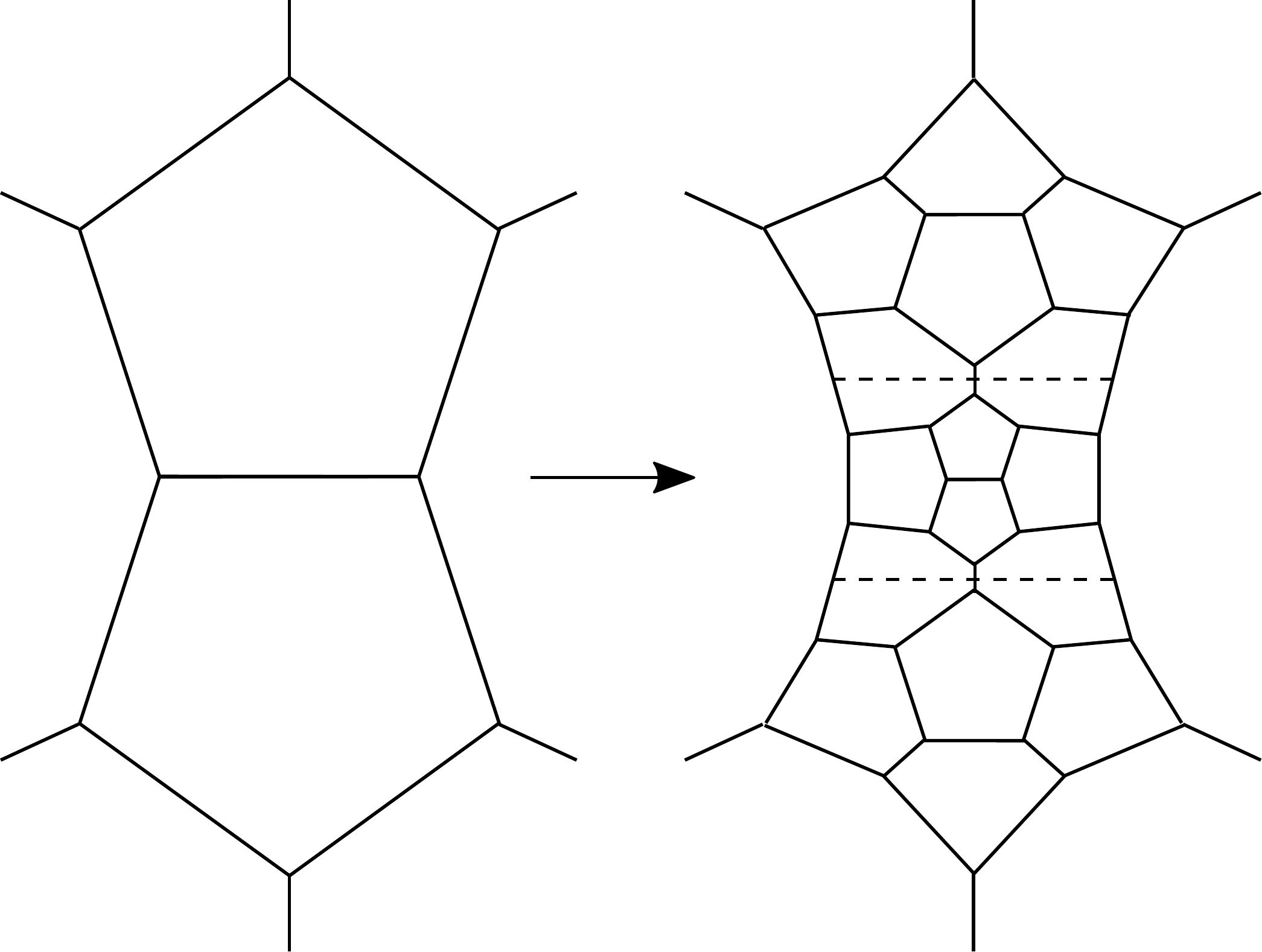}
 \end{center}
 \caption{We attach three dodecahedra to an external edge $e$ separating two boundary pentagons. 
The boundary pattern of the corners changes as shown here. The dashed lines (which are not corners) separate the three dodecahedra.}
 \label{enlargement:fig}
\end{figure}

We do this enlargement in a more symmetric way, not only on the pair of green-dotted faces of $P$, but also simultaneously on 7 more pairs of faces of similar kind, so on 8 pairs overall. We do this to the other pair of adjacent pentagons incident to the face marked with a 5 in Figure \ref{8dodecaedri_new:fig}-(left), because these have the same problem noted above (they become the same face after the identification). We also do it the other 6 pairs of adjacent pentagons obtained from these two by applying all symmetries of $P$ that preserve the yellow octagon. Let $Q$ be the resulting convex hyperbolic polyhedron, tessellated in $8+3\cdot 8 = 32$ dodecahedra. We take two copies of $Q$ and identify the faces as suggested in Figure \ref{8dodecaedri_new:fig}. These identifications were defined for $P$, but they also extend uniquely to $Q$ because $Q$ has been enlarged in a symmetric way. These identifications produce a hyperbolic 3-manifold with corners $N$.

We have constructed our 3-manifold with corners $N$. We now study its topology.
Figure \ref{8dodecaedri_new:fig} shows in each copy of $P$ a red closed curve $\alpha$, contained in the 1-skeleton of the boundary. The cone over $\alpha$ with centre $v$ is a disc $D$ contained in the 2-skeleton of the tessellation of $P$. The disc $D$ is tessellated into 3 pentagons and is pleated at right angles along three orthogonal edges exiting from $v$. The two such discs $D$ and $D'$ in the two copies of $P$ glue to form a torus with one hole $F \subset N$. 

The boundary $\partial F$ lies in the yellow face that is isometric to $\Sigma$, and coincides there with either the curve $\gamma_1$ or $\gamma_2$, depending on the chosen identification between the yellow face and $\Sigma$. We fix once for all an identification that sends $\partial F$ to $\gamma_1$. We orient $N$ coherently with the orientation of $\Sigma$.

It is important to note that $N$ deformation retracts onto $F\cup \Theta$. This can be proved by looking at Figure \ref{8dodecaedri_new:fig}.

\subsection{The 3-manifold with corners $X^{e,i,g}$} \label{sec:X}
We now construct another oriented hyperbolic 3-manifold with corners $X^{e,i,g}$, that depends on the initial parameters $e,i,g$. Recall that $g \geq 2\big(|e|+i+1\big)$.

\begin{figure}
 \begin{center}
 \labellist
\small\hair 2pt
\pinlabel $v$ at 133 65
\endlabellist
  \includegraphics[width = 8 cm]{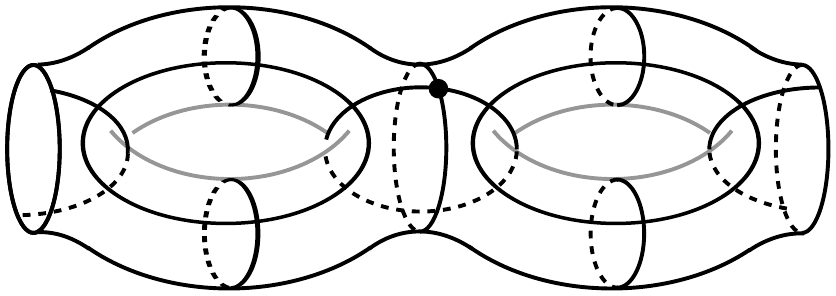}
 \end{center}
 \caption{A hyperbolic surface of genus two with two boundary components, tessellated into 16 right-angled pentagons. The surface $S_{e,i,g}$ contains $i$ portions of this type.}
 \label{isole:fig}
\end{figure}

\begin{figure}
 \begin{center}
 \labellist
\small\hair 2pt
\endlabellist
  \includegraphics[width = 12.5 cm]{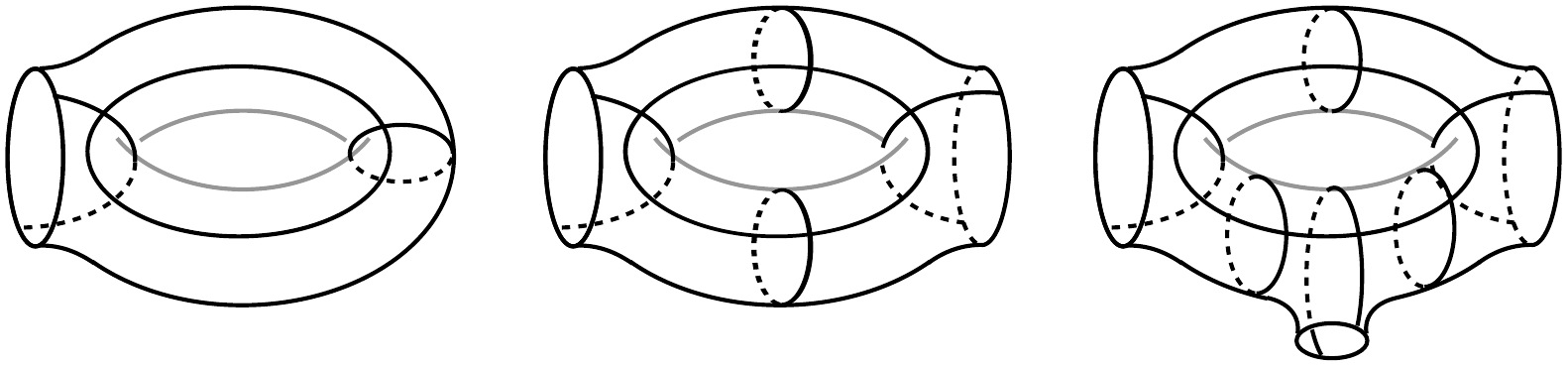}
 \end{center}
 \caption{A hyperbolic surface with $k$ boundary components and genus 1, tessellated into $4k$ right-angled pentagons. We show here the cases $k=1,2,3$, the picture for a general $k$ is easily deduced. }
 \label{Euler:fig}
\end{figure}

We first build a surface $S_{e,i,g}$ tessellated into pentagons, as follows.
We start with a row of $i$ copies of the surface shown in Figure \ref{isole:fig}, to which we attach some pieces as in Figure \ref{Euler:fig} in order to ensure that $S_{e,i,g}$ has exactly $|e|$ boundary components and genus $g-2|e|$. This is possible since $g-2|e|\geq 2(i+1)$. The surface has $\chi(S_{e,i,g}) = 2+3|e|-2g$ and is tessellated into $8g - 12|e| - 8$ right-angled pentagons.

We construct a hyperbolic 3-manifold with corners $X_0^{e,i,g}$ by thickening $S_{e,i,g}$ as described in Section \ref{thickenings:subsection}. 
Clearly $X_0^{e,i,g}$ deformation retracts onto $S_{e,i,g}$. Each of the $|e|$ boundary components of $S_{e,i,g}$ has length $4\ell$ (again, $\ell$ is the edge length of the right-angled pentagon) and is contained in an annular face of $X_0^{e,i,g}$ tessellated into 8 right-angled pentagons. This face is like $\Sigma$ from Figure \ref{sigma_new:fig}, except that the edges labeled with B are not identified. By attaching isometrically two pentagonal faces of $X_0^{e,i,g}$ incident to these two B edges we transform this face into a new face that is isometric with $\Sigma$. We apply one such identification at every boundary component of $X_0^{e,i,g}$. To be more precise, we do this in an orientation-coherent way: we fix an orientation for $S_{e,i,g}$, and choose the B edges coherently everywhere, so that with the induced orientation on $X^{e,i,g}$ these $|e|$ new faces will all be orientation-preservingly isometric to $\Sigma$.

The result is an oriented hyperbolic 3-manifold with corners $X^{e,i,g}$. It has $|e|$ faces $\Sigma_1,\ldots, \Sigma_{|e|}$, each orientation-preservingly isometric to $\Sigma$. It is tessellated into $16g - 24|e| - 16$
right-angled dodecahedra. By construction, it deformation retracts onto $S_{e,i,g} \cup \Theta_1 \cup \cdots \cup \Theta_{|e|}$ where $\Theta_i \subset \Sigma_i$ is the corresponding $\theta$-graph.

\subsection{The skeleton $Y^{e,i,g}$ and the surface $S$} \label{sec:T}
We combine all the objects defined in the previous sections to build a 3-dimensional object $Y^{e,i,g}$ tessellated by right-angled dodecahedra
, called the \emph{skeleton} (since the desired 4-manifold with corners $M^{e,i,g}$ will be a thickening of $Y^{e,i,g}$). 

Recall that $N$ and $X^{e,i,g}$ have respectively one and $|e|$ faces that are orientation-preservingly identified with $\Sigma$, and that $\Sigma$ has an orientation-reversing isometric involution $\varphi$ sending $\gamma_1$ to $\gamma_2$. The skeleton $Y^{e,i,g}$ is constructed by taking $X^{e,i,g}$ and then attaching to each face $\Sigma_j$, $j=1,\ldots, |e|$ two copies $N_j, N_j'$ of $N$ along their faces identified with $\Sigma$. We identify $N_j$ via the identity and $N_j'$ via $\varphi$. 

The skeleton $Y^{e,i,g}$ is a 3-dimensional object. If $e\neq 0$ it is not a manifold, because it is singular at the surfaces $\Sigma_1,\ldots, \Sigma_{|e|}$, to each of which three manifolds with boundary are locally attached.

Let $F_j\subset N_j$, $F_j' \subset N_j'$ be the holed tori introduced in Section \ref{sec:N}. The crucial fact here is that the surfaces $F_j$, $F_j'$, and $S_{e,i,g}$ are attached to the curves $\gamma_1, \gamma_2,$ and $\gamma_0$ of $\Theta_j$, respectively. Therefore
$$S = S_{e,i,g} \cup F_1 \cup F_1' \cup \dots \cup F_{|e|} \cup F_{|e|}'$$
is a surface of genus $g$. For instance, for 
$(e,i,g) = (2,0,6)$ we get a surface $S$ of genus $6$ as in Figure \ref{S_new:fig}. Note that $S_{e,i,g}$ is totally geodesic, while each $F_j$ and $F_j'$ is pleated along some arcs and vertices. 
Another important fact is that by construction the skeleton $Y^{e,i,g}$ deformation retracts onto $S$.

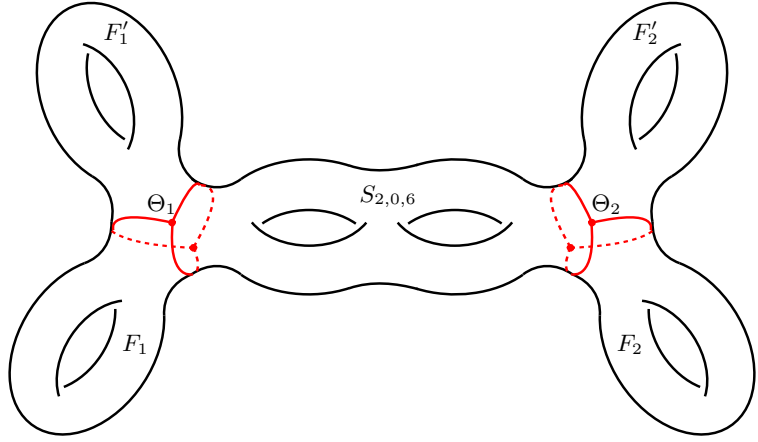
\begin{figure}
 \begin{center}
 \labellist
\small\hair 2pt
\pinlabel $F_1$ at 83 60
\pinlabel $F_1'$ at 71 260
\pinlabel $F_2$ at 400 60
\pinlabel $F_2'$ at 410 260
\pinlabel $S_{2,0,6}$ at 245 155
\pinlabel $\Theta_1$ at 100 150
\pinlabel $\Theta_2$ at 385 150
\endlabellist
  \includegraphics[width = 10 cm]{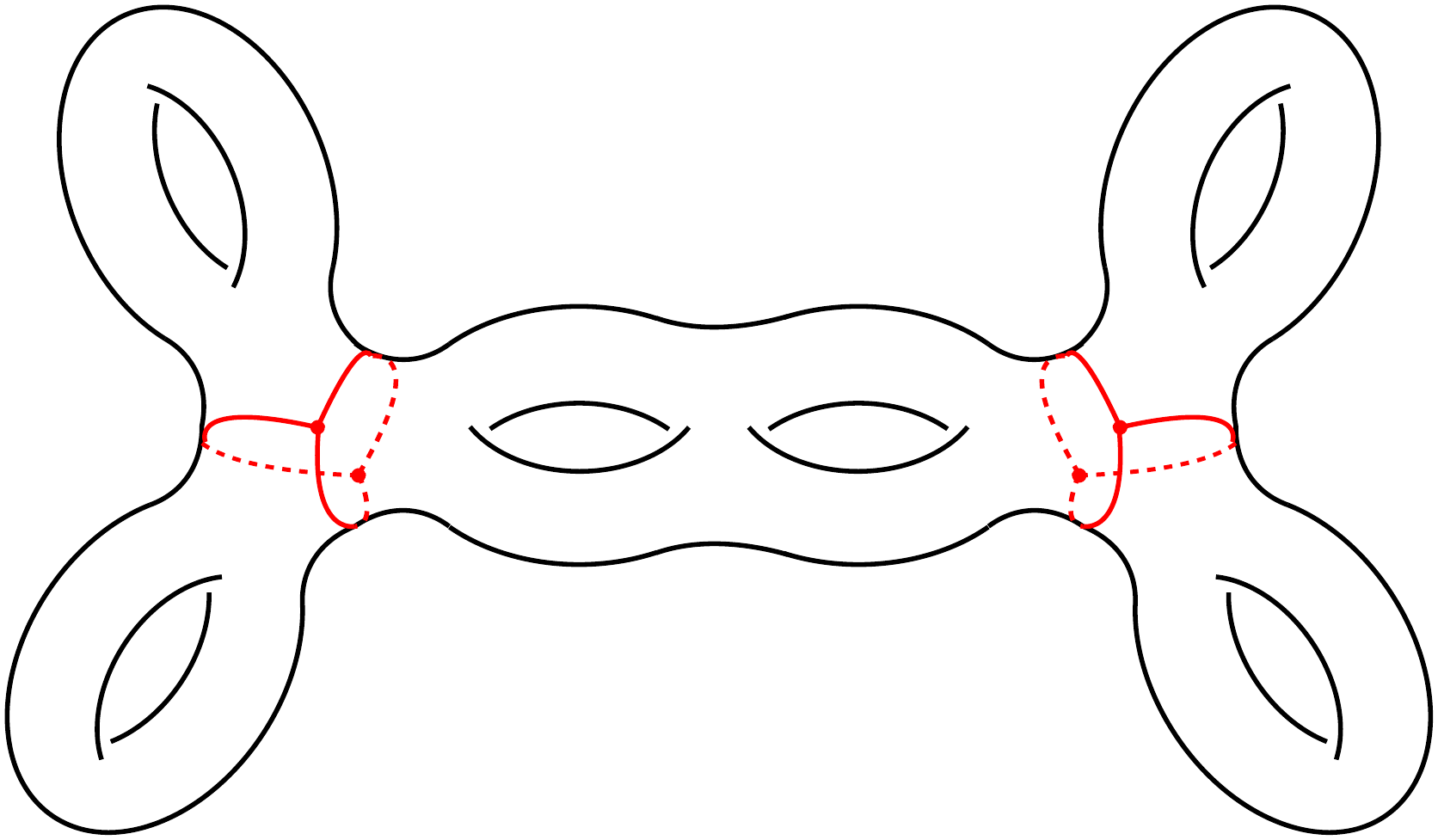}
 \end{center}
 \caption{The surface $S \subset M^{2,0,6}$. }
 \label{S_new:fig}
\end{figure}

\subsection{The 4-manifold with corners $M^{e,i,g}$} \label{sec:M}
We now construct the hyperbolic 4-manifold with corners $M^{e,i,g}$ by appropriately thickening the skeleton $Y^{e,i,g}$, as sketched in Figure \ref{Testendi:fig}-(left). This is done rigorously as follows. For every $j=1,\ldots, |e|$ we consider the oriented hyperbolic 3-manifold with corners $N_j \cup_{\varphi} N_j'$ and thicken it as described in Section \ref{thickenings:subsection}. We also thicken $X^{e,i,g}$. All these thickenings are oriented.

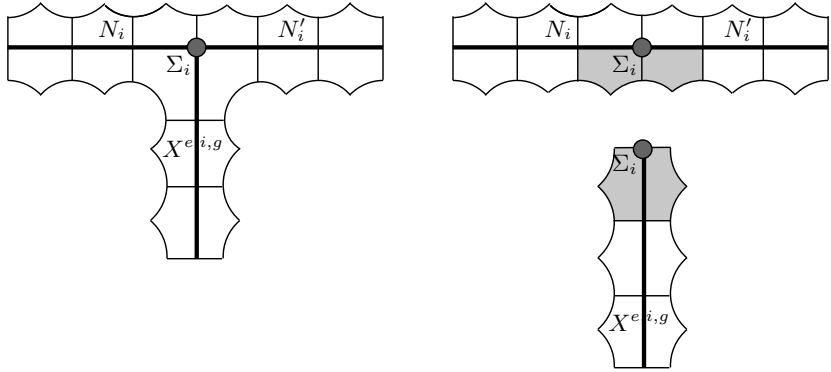
\begin{figure}
 \begin{center}
 \labellist
\small\hair 2pt
\pinlabel $\Sigma_i$ at 97 170
\pinlabel $N_i$ at 60 190
\pinlabel $N_i'$ at 160 190
\pinlabel $X^{e,i,g}$ at 105 125
\pinlabel $\Sigma_i$ at 345 170
\pinlabel $\Sigma_i$ at 345 115
\pinlabel $N_i$ at 308 190
\pinlabel $N_i'$ at 408 190
\pinlabel $X^{e,i,g}$ at 353 28
\endlabellist
  \includegraphics[width = 11 cm]{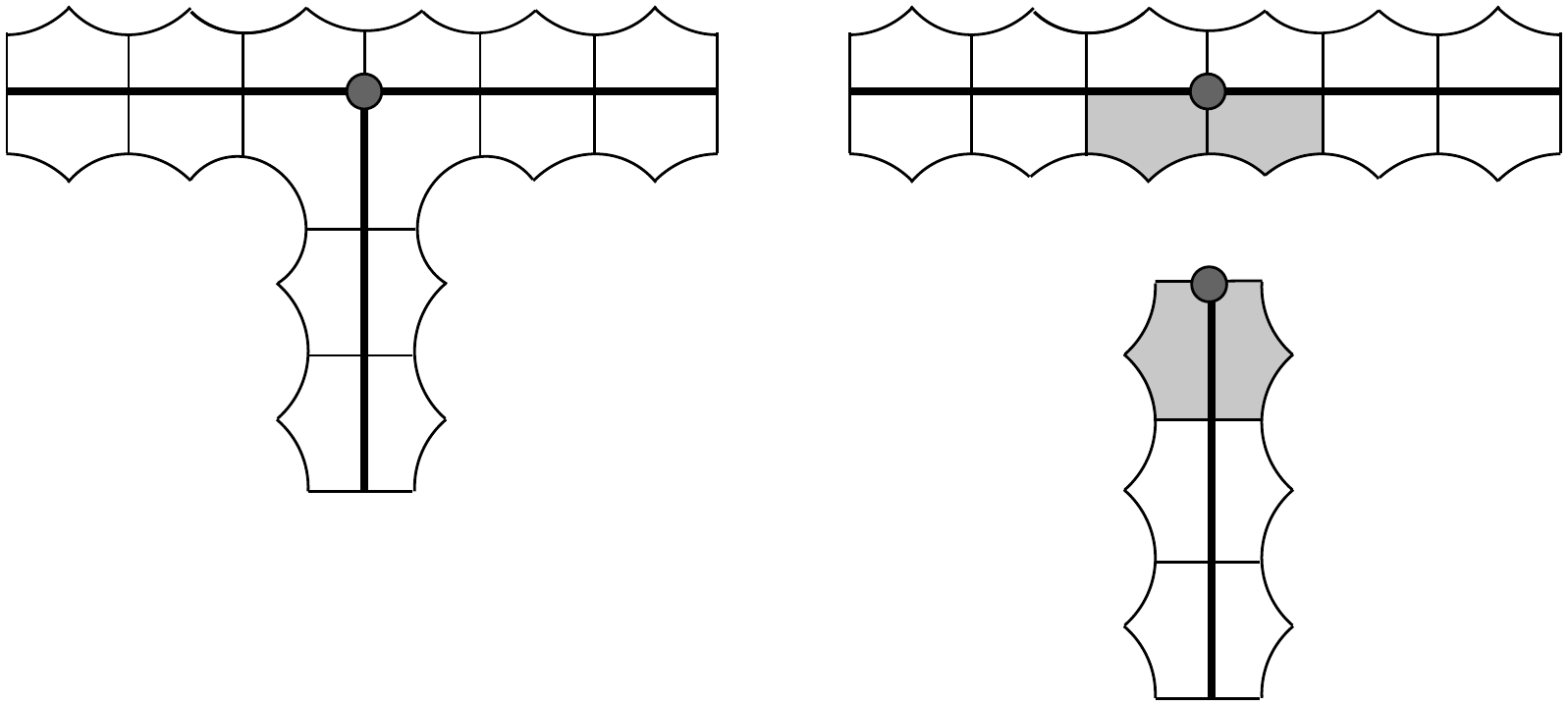}
 \end{center}
 \caption{We thicken the 
skeleton $Y^{e,i,g}$ to a hyperbolic 4-manifold with corners by attaching 120-cells. Here we draw the construction in dimension 2, with segments and pentagons instead of dodecahedra and 120-cells (left). This may be seen rigorously as a two-step procedure, where we first thicken $N_i\cup N_i'$ and $X^{e,i,g}$ separately and then we identify the grey 120-cells (right).}
 \label{Testendi:fig}
\end{figure}

Now we identify in pairs (in an orientation-preserving way) the sixteen 120-cells in the thickening of $X^{e,i,g}$ incident to $\Sigma_j$ with the sixteen 120-cells in the thickening of $N_j \cup_\varphi N_j'$ that are incident to $\Sigma_j$ from below, as in Figure \ref{Testendi:fig}-(right). There is a natural unambiguous way to do this, as suggested by the figure. Note that by construction all the 120-cells involved are distinct. The fact that such an identification produces indeed a manifold with corners (and in particular does not produce any forbidden concave angle of 270 degrees) is due to hyperbolic geometry (this would not be true in Euclidean geometry). The absence of forbidden concave angles was proved in \cite[Lemma 2.5]{MRS} and the same argument applies here.

If we perform this identification for every $j=1,\ldots,|e|$ we get at the end an oriented hyperbolic 4-manifold with corners $M^{e,i,g}$. By construction $M^{e,i,g}$ deformation retracts onto its skeleton $Y^{e,i,g}$, and hence onto $S$. After smoothing its corners, the manifold $M^{e,i,g}$ is a disc bundle over $S$.  

\begin{prop}
The disc bundle $M^{e,i,g}$ over the genus-$g$ surface $S$ has Euler number $\pm e$.
\end{prop}
\begin{proof}
The Euler number may be calculated using a formula of Gromov -- Lawson --  Thurston \cite{GLT} as a sum of contributions of the vertices of $S$. As shown in \cite{MRS}, all the vertices contribute with zero, except two vertices in each $\Theta_j$, that contribute with $\pm\frac 12$ each (with the same sign everywhere). Therefore we get $\pm \frac 12|e|\cdot 2 = \pm e$.
\end{proof}

We could in principle determine the precise sign of the Euler number, but we do not need to do this. We assign once for all to $M^{e,i,g}$ the orientation that gives the bundle the Euler number $e$.

The parameter $i$ is there to ensure that $M^{e,i,g}$ is large enough to be plumbed geometrically simultaneously in $i$ distinct zones. By construction, the surface $S_{e,i,g}$ contains $i$ portions as in Figure \ref{isole:fig}. In each such portion there is a central vertex $v$ as in the figure. The vertex $v$ is adjacent to four distinct embedded pentagons in $S_{e,i,g}$, and then to 16 distinct embedded 120-cells in $M^{e,i,g}$. These 16 distinct 120-cells form altogether a big right-angled polytope $Z$ that is the four-dimensional analogue of the polyhedron $P$ shown in Figure \ref{8dodecaedri_colorato:fig}. We get $Z$ if we pick all the sixteen 120-cells adjacent to a fixed vertex in the tessellation of $\matH^4$ into right-angled 120-cells. We call each of the disjoint $Z_1,\ldots, Z_i \subset M^{e,i,g}$ obtained in this way an \emph{island}. 

\subsection{The plumbing $M$} \label{sec:plumbing}
We can now prove the main part of Theorem \ref{main:teo}. 

\begin{teo}
Let $M$ be a $\partial$-connected sum of plumbings, each whose graph satisfies the inequality
$$g_i \geq 2\big(|e_i| + v_i +1)$$
at every vertex, where $v_i$ is it valence. Then $M$ admits a hyperbolic structure with right-angled corners. 
\end{teo}

\begin{figure}
 \begin{center}
 \labellist
\small\hair 2pt
\endlabellist
  \includegraphics[width = 12 cm]{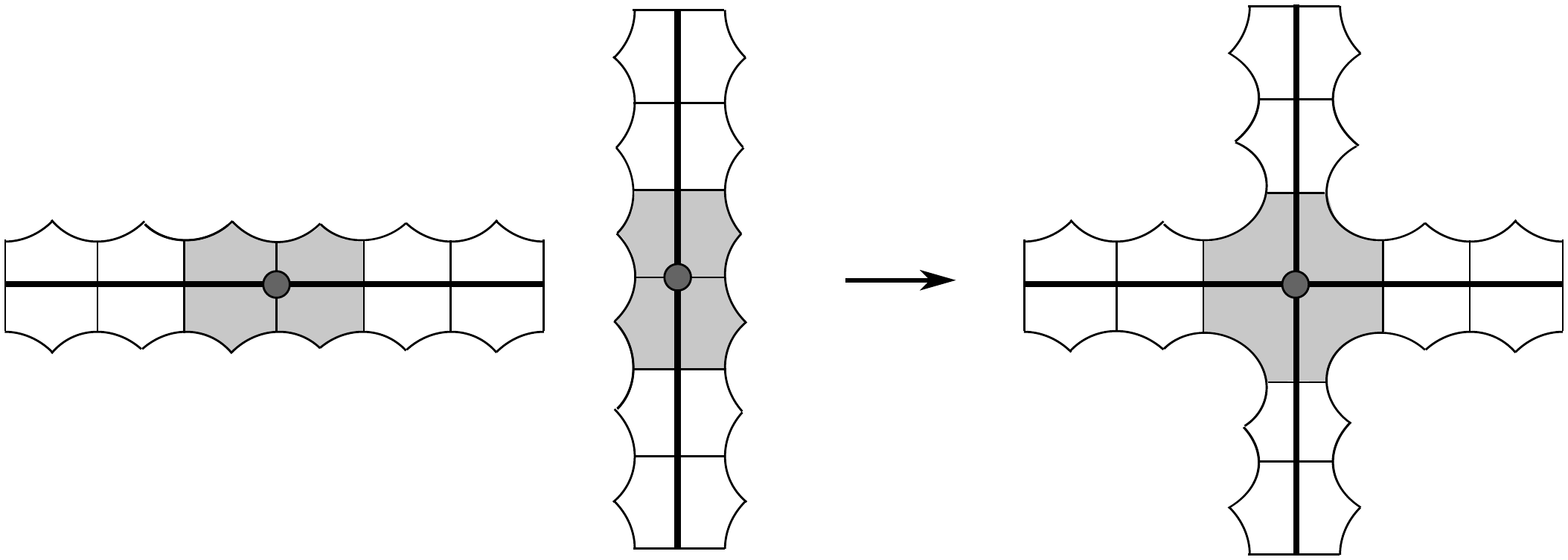}
 \end{center}
 \caption{Two islands (coloured in grey) can be plumbed to produce a new hyperbolic manifold with corners. Here we draw segments and pentagons instead of pentagons and 120-cells.}
 \label{isola:fig}
\end{figure}

\begin{proof}
Consider a plumbing graph satisfying the requirement. For each vertex $j$, pick the hyperbolic manifold with corners $M_j=M^{e_j,v_j,g_j}$, that exists thanks to the assumed inequality. It comes equipped with a genus-$g_j$ surface $S_j \subset M_j$ and $v_j$ disjoint islands. 

Each edge of the plumbing graph is decorated with a sign $\varepsilon = \pm 1$. If it connects the vertices $j$ and $j'$, we choose two islands in $M_j$ and $M_{j'}$, and identify them in a way that makes $S_j$ and $S_{j'}$ intersect orthogonally with the sign $\varepsilon$. See a picture in dimension two in Figure \ref{isola:fig}. We do this for every edge of the plumbing graph.

The result is a new hyperbolic four-manifold with corners. This follows from the fact that the identifications of islands do not produce forbidden concave angles of 270 degrees. To verify this fact, one has to look more closely at the combinatorics of the 120-cell. We perform an analysis analogous to that of \cite[Section 4.3]{MRS}.

Consider the centre $v$ of two 
identified islands, and let $C$ be one of the 120-cells incident to $v$. There is precisely one pentagon $P \subset C$ lying in $S_j$ and one pentagon $P' \subset C$ lying in $S_{j'}$.
These two pentagons intersect each other only in the vertex $v$. Equivalently, they correspond to opposite edges in the tetrahedral vertex figure of $v$ in $C$. There are five distinct dodecahedra in $C$ which intersect the pentagon $P$ in an edge. Of these five dodecahedra, two contain the vertex $v$ and the other three, denoted by $D_k$, $k=1,2,3$ do not. Similarly, there are five dodecahedra which intersect $P'$, and three of these, denoted by $D'_k$, $k=1,2,3$ do not contain the vertex $v$. The crucial property is that no two dodecahedra of the form $D_k$, $D'_l$ are adjacent in $C$. This is sufficient to ensure the absence of a concave angle of 270 degrees, as such phenomenon in our setting can only arise when a 120-cell $C_1$ is glued to $C$ along some dodecahedron $D_k$, and another 120-cell $C_2$ is glued to $C$ along some $D'_l$, with $D_k$ and $D'_l$ adjacent in $C$. We have thus proved that the resulting object is a hyperbolic 4-manifold with right-angled corners.

We have constructed a hyperbolic structure with right-angled corners on every plumbing satisfying the requirements. We may then connect distinct plumbings via geometric $\partial$-connected sums as follows: we select for each connected component an embedded dodecahedral facet, and glue in pairs through arbitrary isometries these facets. Such dodecahedral facets clearly exist, since we left unpaired many facets of each 120-cell in the construction of $M_j$. For instance, consider a pentagon $P$ which lies in $S_j$ but not in an island, choose a 120-cell $C$ of which $P$ is a face, and consider the opposite pentagon $P'$ in $C$. The two dodecaheda adjacent to $P'$ are embedded.

The proof is complete by Proposition \ref{mirror:prop}. 
\end{proof}

\begin{figure}
 \begin{center}
 \labellist
\small\hair 2pt
\pinlabel $(1)$ at 250 300
\pinlabel $(2)$ at 580 300
\pinlabel $(3)$ at 250 20
\pinlabel $(4)$ at 580 20
\endlabellist
  \includegraphics[width = 8 cm]{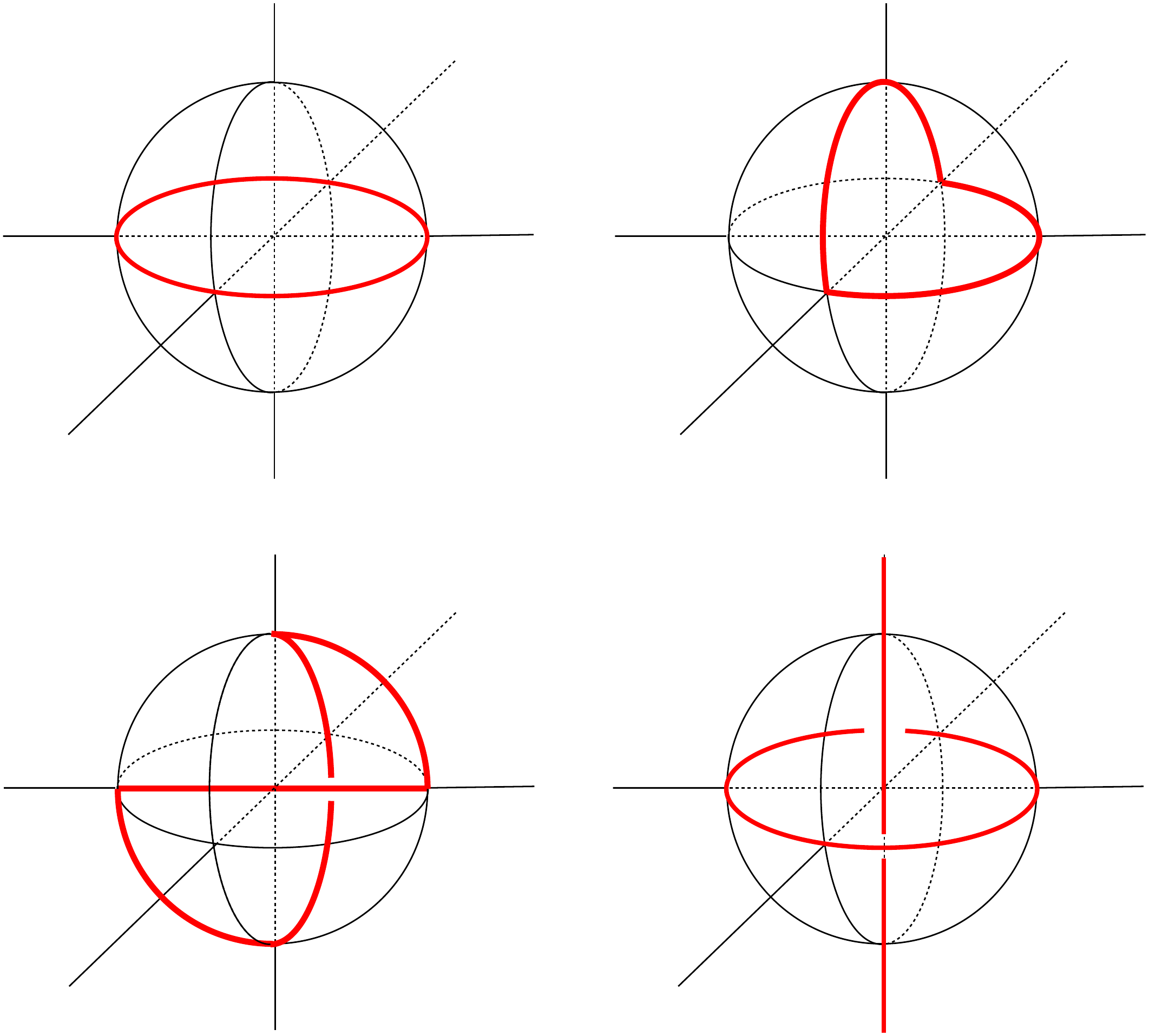}
 \end{center}
 \caption{The immersed base surface in $M$ have vertices with four types of links. The vertices of type (3) are those of $\Theta_j$ and contribute to the Euler numbers. Those of type (4) are the self-intersections.}
 \label{links:fig}
\end{figure}
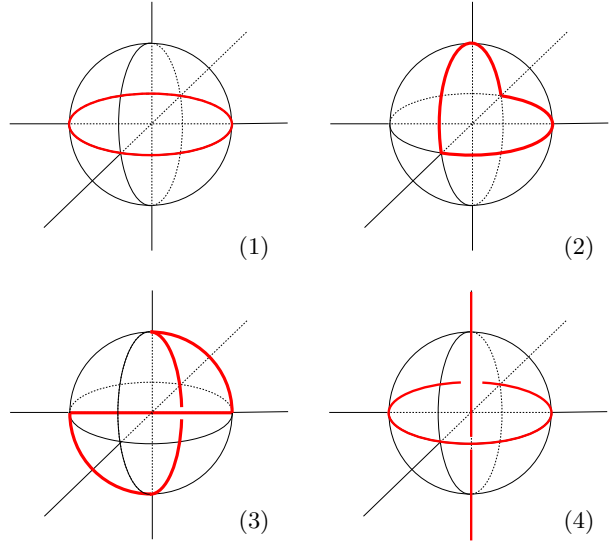

\begin{rem}
By construction, each manifold with corners $M^{e,i,g}$ is tessellated into right-angled 120-cells and the base surface $S\subset M^{e,i,g}$ is contained in its 2-skeleton. The surface $S$ is tessellated into right-angled pentagons, that are pleated along some edges and vertices. The pleating can be described locally as follows. Every vertex $v\in S$ is adjacent to sixteen 120-cells of the tessellation, whose link at $v$ is the standard triangulation of $S^3$ into sixteen right-angled tetrahedra. The link of $S$ at $v$ is contained in the 1-skeleton of such a triangulation of $S^3$. By construction, the vertices of $S$ have three possible types of links, shown in Figure \ref{links:fig}-(1, 2, 3). These contribute to the self-intersection $S\cdot S$ respectively with 0, 0, and $\pm \frac 12$, via a formula of \cite{GLT}. The vertices of type (3) are precisely those of $\Theta_j$. See also \cite{MRS}.

All these surfaces form altogether an immersed surface in $M$, still contained in the 2-skeleton, onto which $M$ deformation retracts. The vertices where the immersed surface self-intersect have the Hopf link shown in Figure \ref{links:fig}-(4).
\end{rem}

\subsection{Volume estimates} \label{sec:volume}
We conclude here the proof of Theorem \ref{main:teo} by computing the volume of $M$ and estimating that of $W$.

By construction, the skeleton $Y^{e,i,g}$ is tessellated into $y_{e,g}$ dodecahedra, where
$$y_{e,g} \ = \ 2 \cdot 64 |e| + 16g - 24|e| - 16 \ = \ 16g + 104 |e| -16.$$
The addendum $16g - 24|e| - 16$ is simply the number of dodecahedra of the $3$-manifold with corners $X^{e,i,g}$. The other addendum arises because we glue to each of the $|e|$ boundary components of $X^{e,i,g}$ two copies of the $3$-manifold with corners $N$, with each copy consisting of $64$ dodecahedra.

So $M^{e,i,g}$ is tessellated into $m_{e,g}$ 120-cells, where
$$m_{e,g} \ = \ 2 y_{e,g} - 16 |e| \ = \ 32g + 192 |e| - 32.$$ This is true because the thickening of $Y^{e,i,g}$ to $M^{e,i,g}$ can be seen as a two-step procedure: we first thicken each of $X^{e,i,g}$ and the $|e|$ copies of $N$ separately (thus obtaining a total of $2 y_{e,g}$ distinct 120-cells) and then identify sixteen 120-cells from the thickening of each copy of $N$ to sixteen 120-cells from the thickening of $X^{e,i,g}$ as in Figure \ref{Testendi:fig}. 

Let now $V$ and $E$ be the number of vertices and edges of the plumbing graph, respectively. Then $M$ is tessellated into $m$ distinct 120-cells, where
$$m \ = \ \sum_j m_{e_j,g_j} - 16 E \ = \ \sum_j \left( 32g_j + 192 |e_j| \right) - 32 V - 16 E.$$
Again, this holds true because each of the $E$ plumbings which we perform as in Figure \ref{isola:fig} identifies two distinct sets of sixteen 120-cells from the various $M^{e,i,g}$'s into a single one.

Since the right-angled 120-cell has volume $34\pi^2/3$, we have
$$\Vol(M)\ =\ \frac{34\pi^2}{3} m\ \leq\ \frac{34\pi^2}3 \sum_j 128 g_j\ \approx\ 14317.50 \sum_j g_j.$$
The inequality follows from $|e_j| \leq g_j/2$, which is true since $g_j  \geq  2(|e_j| + v_j + 1)$ by hypothesis. We have proved the first inequality of Theorem \ref{main:teo} with $C_1= 14318.$ 

To estimate the volume of $W$, we notice that the same argument of \cite[Lemma 7]{M} applies in this setting. We briefly recall it.
%

It is not difficult to show that the number of isometry classes of the facets of $M$ is bounded by a constant that does not depend on the plumbing graph. In particular, each facet of $M$ has at most $f$ faces, for some universal $f$. (A willing reader may explicitly compute or estimate $f$.) In other words, each vertex of the adjacency graph of the facets of $M$ has valence at most $f$. Since every finite graph without loops and with valence $\leq k$ can be vertex-coloured with at most $k+1$ colours, we can colour the facets of $M$ with at most $f+1$ colours. 

By the proof of Proposition \ref{mirror:prop}, we thus can choose $W$ so that
$$\Vol(W) \leq 2^{f+1} \cdot \Vol(M),$$
and the proof of Theorem \ref{main:teo} is complete with $C_2=2^{f+1}$.

\end{document}